\documentclass {article}
\usepackage{dsfont}
\usepackage{bbold}
\usepackage{amsfonts}
\usepackage{amsthm}
\usepackage{amsmath}
\usepackage[margin=2.8cm]{geometry}
\usepackage{mathpazo}
\usepackage[scaled=.95]{helvet}

\usepackage[]{algorithm2e}
\usepackage{graphicx} 
\usepackage{multirow}
\usepackage[utf8]{inputenc}
\usepackage{lmodern}
\newtheorem{prop}{Proposition}[section]
\newtheorem{lem}{Lemma}[section]
\newtheorem{theo}{Theorem}[section]

\usepackage[colorlinks,allcolors=blue]{hyperref}
\usepackage{natbib}

\begin{document}

\title{A penalized criterion for selecting the number of clusters for K-medians}
\author{Antoine Godichon-Baggioni and Sobihan Surendran, \\
        Laboratoire de Probabilités, Statistique et Modélisation\\
       Sorbonne-Université,       75005 Paris, France \\       
        antoine.godichon$\_$baggioni@upmc.fr, sobihan.surendran@sorbonne-universite.fr  \\}

\date{}
\maketitle

\begin{abstract}
Clustering is a usual unsupervised machine learning technique for grouping the data points into  groups based upon similar features. We focus here on unsupervised clustering for contaminated data, i.e in the case where K-medians algorithm should be preferred to K-means because of its robustness. More precisely, we concentrate on a common question in clustering: how to chose the number of clusters? The answer proposed here is to consider the choice of the optimal number of clusters as the minimization of a penalized criterion. In this paper, we obtain a suitable penalty shape for our criterion and derive an associated oracle-type inequality. Finally, the performance of this approach with different types of K-medians algorithms is compared on a simulation study with other popular techniques. All studied algorithms are available in the R package \texttt{Kmedians} on CRAN.
\end{abstract}


\section{Introduction}

Clustering is an unsupervised machine learning technique that involves grouping data points into a collection of groups based on similar features. Clustering is commonly used for data compression in image processing, which is also known as vector quantization \citep{gersho2012vector}. There is a vast literature on clustering techniques, and general references regarding clustering can be found in \cite{spath1980cluster, jain1988algorithms, mirkin1996mathematical, jain1999data, berkhin2006survey, kaufman2009finding}. Classification methods can be categorized as hard clustering also referred as crisp clustering (including K-means, K-medians, and hierarchical clustering) and soft clustering (such as Fuzzy K-means \citep{dunn1973fuzzy, bezdek2013pattern} and Mixture Models). In hard clustering methods, each data point belongs to only one group, whereas in soft clustering, a probability or likelihood of a data point belonging to a cluster is assigned, allowing each data point to be a member of more than one group.

%They are different types of clustering methods like hierarchical clustering, which consists in grouping data into a tree of clusters (one can refer to AGNES (AGglomerative NESting) or DIANA (DIvise ANAlysis) \citep{kaufman2009finding} for instance). Some other usual methods are Density-based clustering which relies on estimated cluster density to perform the partitioning (see DBSCAN \citep{ester1996density} or Mean Shift \citep{cheng1995mean} for instance) and Model-based clustering which considers the data as coming from a mixture of underlying probability distributions such as Expectation-Maximization algorithm \citep{dempster1977maximum}. 

We focus here on  hard clustering methods. The most popular partitioning clustering methods are the non sequential  \citep{forgy1965cluster}  and the sequential  \citep{macqueen1967classification}   versions of the K-means algorithms.  The aim of the K-means algorithm is to minimize the sum of squared distances between the data points and their respective cluster centroid.
More precisely, considering $X_1,...,X_n$ random vectors taking values in $\mathbb{R}^d $, the aim is to find k centroids $\left\{c_1,...,c_k \right\}$ minimizing the empirical distortion 
\begin{equation}
\label{1}
\frac{1}{n}  \sum_{i=1}^{n}  \min_{j = 1,..,k}  \left\|  X_i - c_j \right\|^2 . 
\end{equation}

Nevertheless, in many real-world applications, data collected may be contaminated with outliers of large magnitude, which can make traditional clustering methods such as K-means sensitive to their presence. As a result, it is necessary to use more robust clustering algorithms that produce reliable outcomes. One such algorithm is K-medians clustering, which was introduced by \cite{macqueen1967classification} and further developed by \cite{kaufman2009finding}. Instead of using the mean to determine the centroid of each cluster, K-medians clustering uses the geometric median. It consists in considering criteria based on least norms instead of least squared norms. More precisely, considering the same sequence of i.i.d copies $X_1,...,X_n$, the objective of K-medians clustering is to minimize the empirical $L^1$-distortion :
$$ \frac{1}{n}  \sum_{i=1}^{n}  \min_{j = 1,..,k}  \left\|  X_i - c_j \right\|.  $$ 
\\
In practical applications, the number of clusters k is often unknown. In this paper, we will focus on the choice of optimal number of clusters for robust clustering. Several methods for determining the optimal number of clusters have been studied for K-means algorithms and can be easily adapted for K-medians. One commonly used method for determining the optimal number of clusters is the elbow method. Other methods often used are the Silhouette \citep{kaufman2009finding} and the Gap Statistic \citep{tibshirani2001estimating}. The silhouette coefficient of a sample is defined as the difference between the within-cluster distance between the sample and other data points in the same cluster and the inter-cluster distance between the sample and the nearest cluster. The Silhouette method suggests selecting the value of k that maximizes the average silhouette coefficient of all data points. The silhouette score is typically calculated using Euclidean or Manhattan distance. Regarding the Gap Statistic, the idea is to compare the within-cluster dispersion to its expected value under an appropriate null reference distribution. The reference data set is generated via Monte Carlo simulations of the sampling process.
\\

In \cite{fischer2011number}, the objective is to minimize the empirical distortion, which is defined in (1), as a function of k in order to determine the optimal number of clusters. However, if all data points are placed in a single cluster, the empirical distortion will be minimized. To prevent choosing too large a value for k, a penalty function is introduced. It was shown that the penalty shape is $\sqrt{\frac{k}{n}}$ in the case of K-means clustering and by finding the constant of the penalty with the data-based calibration method, one can obtain better results than by using usual other methods. The data-driven calibration algorithm is a method proposed by \cite{birge2007minimal}  and developed by \cite{arlot2009data} , to find the constant of penalty function. Theoretical properties on this data-based penalization procedures have been studied by \cite{birge2007minimal, arlot2009data, baudry2012slope}. The aim of this paper is to adapt these methods for K-medians algorithms. We first provide the shape of the penalty function and then use the slope heuristic method to calibrate the constant and construct a penalized criterion for selecting the number of clusters for K-medians algorithms.
\\

The paper is organized as follows. We provide a recap of two different methods for estimating the geometric median, followed by the introduction of three K-median algorithms (“Online”, “Semi-Online”, and “Offline”). In section \ref{sec::choice}, we propose a penalty shape for the proposed penalized criterion and give an upper bound for the expectation of the distortion at empirically optimal codebook with size of optimal number of clusters which ensure our penalty function. We illustrate the proposed approach with some simulations and compare it with several methods in section \ref{sec::simu}. Finally, the proofs are gathered in section \ref{sec::proofs}. All the proposed algorithms are available in the R package \texttt{Kmedians} on CRAN \url{https://cran.r-project.org/package=Kmedians}.

\section{Framework}\label{sec::framework}

\subsection{Geometric Median}

In what follows, we consider a random variable X that takes values in $\mathbb{R}^d$ for some $d \ge 1$.
It is well-known that the standard mean of X is not robust to corruptions. Hence, the median is preferred to the mean in robust statistics. The geometric median $m$, also called $L^1$-median or spatial median,  of a random variable $X\in \mathbb{R}^d$ is defined by  \cite{haldane1948note} as follows: 
$$ m = \arg\min_{u \in \mathbb{R}^d} \mathbb{E} \left[ \left\|  X - u \right\| \right]. $$

For the 1-dimensional case, the geometric median coincides with the usual median in $\mathbb{R}$. As Euclidean space $\mathbb{R}^d$ is strictly convex, the geometric median $m$ exists and is unique if the points are not concentrated around a straight line \citep{kemperman1987median}. The geometric median is known to be robust and has a breakdown point of $0.5$.

Let us now consider a sequence of i.i.d copies $X_1,...,X_n$ of $X$. In this paper, we focus on two methods to determine the geometric median. The first one is iterative and consists in considering the fix point estimates \citep{weiszfeld1937point, vardi2000multivariate}
\[ 
\hat{m}_{t+1} = \frac{ \sum_{i \in \mathcal{X}_t }  \frac{X_i}{ \left\|  X_i - \hat{m}_t \right\| } }{ \sum_{i \in \mathcal{X}_t }  \frac{1}{ \left\|  X_i - \hat{m}_t \right\| }  }   
\]
with a initial point $\hat{m}_0 \in \mathbb{R}^d$ chosen arbitrarily such that it does not coincide with any of the $X_i$ and $\mathcal{X}_t = \{i, X_i \ne \hat{m}_{t} \}$.
This Weiszfeld algorithm can be a flexible technique, but there are many implementation difficulties for massive data in high-dimensional spaces.

An alternative and simple estimation algorithm which can be seen as a stochastic gradient algorithm  \citep{robbins1951stochastic, ruppert1985newton, duflo1997random, cardot2013efficient} and is defined as follows
$$ m_{j+1} = m_j + \gamma_j  \frac{X_{j+1} - m_j}{ \left\|  X_{j+1} - m_j \right\| }   $$
where $m_0$ is an arbitrarily chosen starting point and $\gamma_j$ is a step size such that $\forall j \ge 1, \gamma_j > 0, \sum_{j \ge 1} \gamma_j = \infty $ and $ \sum_{j \ge 1} \gamma_j^2 < \infty $.
Its averaged version (ASG), which is effective for large samples of high dimension data,  introduced by  \cite{polyak1992acceleration}  and adapted by  \cite{cardot2013efficient}, is defined by
$$ \overline{m}_{j+1} = \overline{m}_j + \frac{1}{j+1}  (m_{j+1} - \overline{m}_j).  $$
One can speak about averaging since $ \overline{m}_j = \frac{1}{j} \sum_{i=1}^{j} m_{i} $.
We note that, under suitable assumptions, both $ \hat{m}_{t}$ and $\overline{m}_{t}$ are asymptotically efficient \citep{vardi2000multivariate, cardot2013efficient}.
\\

\subsection{K-medians} 

For a positive integer k, a vector quantizer Q of dimension d and codebook size k is a (measurable) mapping of the d-dimensional Euclidean $\mathbb{R}^d $ into a finite set of points $\left\{c_1,...,c_k \right\}$ \citep{linder2000training}. More precisely,
the points $c_i \in \mathbb{R}^d, i = 1,...,k$ are called the codepoints and the vector composed of the code points $\left\{c_1,...,c_k \right\}$ is called codebook, denoted by c. Given a d-dimensional random vector X admitting a finite first order moment, the $L^1$-distortion of a vector quantizer Q with codebook $c = \left\{c_1,...,c_k \right\}$ is defined by
\begin{equation} \label{W}
W(c) :=  \mathbb{E} \left[ \min_{j = 1,..,k}  \left\|  X - c_j \right\|   \right].
\end{equation}

Let us now consider $X_1,...,X_n$ random vectors $\in \mathbb{R}^d$ i.i.d with the same law as X. Then, one can define the empirical $L^1$-distortion as :
\begin{equation}  \label{Wn}
W_n(c) :=  \frac{1}{n}  \sum_{i=1}^{n}  \min_{j = 1,..,k}  \left\|  X_i - c_j \right\|.  
\end{equation}

In this paper, we consider two types of K-medians algorithms : sequential and non sequential algorithm. The non sequential algorithm uses Lloyd-style iteration which alternates between an expectation (E) and maximization (M) step and is precisely described in Algorithm \ref{algo:ssg}:

\vspace{3mm}

\begin{algorithm}[H]
\caption{Non Sequential K-medians Algorithm \label{algo:ssg}.} 
\SetAlgoLined
\DontPrintSemicolon
\SetKwComment{Comment}{/* }{ */}
\SetKwInOut{Inputs}{Inputs}
\SetKwInOut{Output}{Output}
\Inputs{$D = \left\{x_1,...,x_n \right\}$ datapoints, k number of clusters}
\Output{A set of k clusters : $C_1,...,C_k$}
Randomly choose k centroids : $m_1,...,m_k$.\;
\While{the clusters change}{
\For{$1 \le i \le n$}{
$ r = \arg\min_{1 \le j \le k}  \left\|  x_i - m_j \right\| $ \;
$ C_r \leftarrow x_i $ \;
    }
\For{$1 \le j \le k$}{
$ m_j =  \arg\min_{m} \sum_{i, x_i \in C_j } \left\|  x_i - m \right\| $ \;
    }
}
\end{algorithm}

\vspace{1mm}

For $1 \le j \le k, m_j$ is nothing but the geometric median of the points in the cluster $C_j$. As $m_j$ is not explicit, we will use Weiszfeld (indicated by “Offline”) or ASG (indicated by “Semi-Online”) to estimate it. The Online K-median algorithm proposed by   \cite{cardot2012fast}   based on an averaged Robbins-Monro procedure  \citep{robbins1951stochastic,polyak1992acceleration}  is described in Algorithm \ref{algo:ssg2}:

\vspace{5mm}

\begin{algorithm}[H]
\caption{Online K-medians Algorithm \label{algo:ssg2}.} 
\SetAlgoLined
\DontPrintSemicolon
\SetKwComment{Comment}{/* }{ */}
\SetKwInOut{Inputs}{Inputs}
\SetKwInOut{Output}{Output}
\Inputs{$D = \left\{x_1,...,x_n \right\}$ datapoints, k number of clusters, $c_{\gamma}>0$ and $\alpha \in (1/2,1)$}
\Output{A set of k clusters : $C_1,...,C_k$}
Randomly choose k centroids : $m_1,...,m_k$.\;
$ \overline{m}_j = m_j \hspace{1mm}  \forall \hspace{1mm} 1 \le j \le k$\;
$ n_j = 1 \hspace{1mm}  \forall \hspace{1mm} 1 \le j \le k$\;
\For{$1 \le i \le n$}{
$ r = \arg\min_{1 \le j \le k}  \left\|  x_i - \overline{m}_j \right\| $ \;
$ C_r \leftarrow x_i $ \;
$ m_r \leftarrow m_r + \frac{c_{\gamma}}{\left( n_{r}+1 \right)^{\alpha}}  \frac{x_i - m_r}{ \left\|  x_i - m_r \right\| }   $ \;
$ \overline{m}_r \leftarrow \frac{n_r \overline{m}_r + m_r}{n_r + 1} $ \;
$ n_r \leftarrow n_r + 1$ \;
    }
\end{algorithm}

\vspace{1mm}

The non-sequential algorithms are effective but the computational time is huge compared to the sequential (“Online”) algorithm, which is very fast and only requires $O(knd)$ operations, where n is the sample size, k is the number of clusters and d is dimension. Furthermore, in case of large samples, Online algorithm is expected to estimate the centers of the clusters as well as the non-sequential algorithm \cite{cardot2012fast}. Then, in case of large sample size, Online algorithm should be preferred and vice versa.

\section{The choice of k}\label{sec::choice}

In this section, we adapt the results that have been shown for K-means in \cite{fischer2011number} to K-medians clustering. In this aim, let $X_1,...,X_n$ random vectors with the same law as $X$, and we assume that $\| X \| \le R$ almost surely for some $R>0$. Let $S_k$ denote the countable set of all $\left\{c_1,...,c_k \right\} \in \mathbb{Q}^k$, where $\mathbb{Q}$ is some grid over $\mathbb{R}^d$. It is important to note that $\mathbb{Q}$ represents the search space for the centers. Since $\| X\|$ is assumed to be bounded by $R$, we  consider a grid $\mathbb{Q} \subset \overline{\mathcal{B}}(0,R)$ (where $\overline{\mathcal{B}}(0,R)$ denotes the closed ball centered at $0$ with radius $R$. A codebook $\hat{c}_k$ is said empirically optimal codebook if we have $ W_n(\hat{c}_k) = \min_{c \in S_k } W_n(c) $. Let $\hat{c}_k$ be a minimizer of the criterion $W_n(c)$ over $S_k$. Our aim is to determine $\hat{k}$ minimizing a criterion of the type
$$ \text{crit}(k) = W_n(\hat{c}_{k}) + \text{pen}(k) $$
where pen : $\left\{1,...,n\right\} \rightarrow \mathbb{R}_+$ is a penalty function described later. The purpose of this penalty method is to prevent choosing too large a value for k by introducing a penalty into the objective function.\\

In this section, we will give an upper bound for the expectation of the distortion at empirically optimal codebook with size of optimal number of clusters which is based on a general non asymptotic upper bound for 
$$ \mathbb{E} \left[\sup_{c \in S_k } \left\{ W(c) - W_n(c) \right\}  \right] . $$

\begin{theo}\label{theo1} Let $X_1, \ldots, X_n$ be random vectors taking values in $\mathbb{R}^d$ with the same law as $X$, and we assume that $\| X \| \le R$ almost surely for some $R > 0$. Define $W$ and $W_n$ as in \eqref{W} and \eqref{Wn}, respectively. Then for all $1 \le k \le n$, 
$$ \mathbb{E} \left[\sup_{c \in S_k } \left\lbrace W(c) - W_n(c) \right\rbrace  \right]  \le 48R\sqrt{\frac{kd}{n}}. $$
\end{theo}

This theorem shows that the maximum difference of the distortion and the expected empirical distortion of any vector quantizer is of order $n^{-1/2}$. Selecting the search space for the centers is crucial because a larger search space results in a higher upper bound.
\\
\begin{theo}\label{theo2}
Let $X$ be a random vector taking values in $\mathbb{R}^d$ and we assume that $\| X \| \le R$ almost surely for some $R > 0$.
Consider nonnegative weights $ \left\{ x_k \right\}_{1 \le k \le n} $ such that $ \sum_{k=1}^{n} e^{-x_k} =   \Sigma$. Define $W$ as in \eqref{W} and suppose that for all $ 1 \le k \le n $
$$ \text{pen}(k) \ge R\left( 48\sqrt{\frac{kd}{n} }  +  2\sqrt{\frac{x_k}{2n} } \right). $$
Then: 
$$  \mathbb{E} \left[W(\tilde{c})  \right] \le \inf_{ 1 \le k \le n }\left\{ \inf_{c \in S_k }W(c) + \text{pen}(k) \right\} + \Sigma R \sqrt{\frac{\pi}{2n} },  $$
where $\tilde{c} = \hat{c}_{\hat{k}}$ minimizer of the penalized criterion.
\end{theo}

We remark the presence of the weights $ \left\{ x_k \right\}_{1 \le k \le n} $ in penalty function and $\Sigma$ which depends on the weights in upper bound for the expectation of the distortion at  $\tilde{c}$. The larger the weights $ \left\{ x_k \right\}_{1 \le k \le n} $, the smaller the value of $\Sigma$. So, we have to make a compromise between these two terms. 
Let us indeed consider the simple situation where one can take $ \left\{ x_k \right\}_{1 \le k \le n} $ such that $x_k$ = Lk for some positive constant L and $\Sigma = \sum_{k=1}^{n} e^{-x_k} \le 1$.
If we take
$$ \text{pen}(k) = R\left( 48\sqrt{\frac{kd}{n} }  +  2\sqrt{\frac{Lk}{2n} } \right) = R\sqrt{\frac{k}{n}} \left( 48\sqrt{d}  +  2\sqrt{\frac{L}{2} } \right), $$
we deduce that the penalty shape is $a\sqrt{\frac{k}{n}}$ where $a$ is a constant.
\\
\begin{prop}\label{prop}
Let $X$ be a d-dimensional random vector such that $\| X \| \le R$ almost surely. Then for all $1 \le k \le n$,
$$ \inf_{c \in S_k }W(c) \le 4Rk^{-1/d}, $$
where $W$ is defined in \eqref{W}.
\end{prop}

Assume that for every $ 1 \le k \le n $
$$ \text{pen}(k) = aR\sqrt{\frac{k}{n}},  $$
where $a$ is a positive constant that satisfies $a \ge \left( 48\sqrt{d}  +  2\sqrt{\frac{L}{2} } \right)$ to verify the hypothesis of Theorem \ref{theo2}. Using Theorem \ref{theo2} and Proposition \ref{prop}, we obtain:
$$  \mathbb{E} \left[W(\tilde{c})  \right] \le R\left(\inf_{ 1 \le k \le n }\left\{4k^{-1/d} + a\sqrt{\frac{k}{n} } \right\} + \Sigma  \sqrt{\frac{\pi}{2n} }\right).  $$

Minimizing the term on the right hand side of previous inequality leads to k of the order $n^{\frac{d}{d+2}}$ and 

$$  \mathbb{E} \left[W(\tilde{c})  \right]  = \mathcal{O}(n^{-\frac{1}{d+2}}). $$

We conclude that our penalty shape is $a\sqrt{\frac{k}{n}}$ where $a$ is a constant.  In \cite{birge2007minimal},   a data-driven method has been introduced to calibrate such criteria whose penalties are known up to a multiplicative factor: the “slope heuristics”. This method consists of estimating the constant of penalty function by the slope of the expected linear relation of $-W_n(\hat{c}_{k})$ with respect to the penalty shape values $\text{pen}_{\text{shape}}(k) =  \sqrt{\frac{k}{n}}$.\\

\textbf{Estimation of constant $a$: } Let denote $ c^* = \arg\min_{c \in S}  W(c) $ and $ c_k = \arg\min_{c \in S_k}  W(c) $, where S any linear subspace of $\mathbb{R}^d$  and $S_k$ set of predictors (called a model). It was shown in \cite{birge2007minimal, arlot2009data, baudry2012slope}  that under conditions, the optimal penalty verifies for large $n$:
$$\text{pen}_{\text{opt}}(k) := a_{\text{opt}}\text{pen}_{\text{shape}}(k) \approx  2(W_n(c^*) - W_n(\hat{c}_{k})).$$
This gives
$$\frac{a_{opt}}{2}\text{pen}_{\text{shape}}(k) - W_n(c^*) \approx  - W_n(\hat{c}_{k}).$$
The term $-W_n(\hat{c}_{k})$ with respect to the penalty shape behaves like a linear function for a large $k$. The slope $\hat{S}$ of the linear regression of $-W_n(\hat{c}_{k})$ with respect to $pen_{shape}(k)$ is computed to estimate $\frac{a_{opt}}{2}$.
Finally, we obtain $$\text{pen}(k) := a_{\text{opt}}\text{pen}_{\text{shape}}(k) = 2\hat{S}\text{pen}_{\text{shape}}(k).$$
Of course, since this method is based on asymptotic results, it can encounter some practical problems when the dimension $d$ is larger than the sample size $n$.

\section{Simulations}\label{sec::simu}

This whole method is implemented in $\textbf{\textsf{R}}$ and all these studied algorithms are available in the $\textbf{\textsf{R}}$ package \texttt{Kmedians} \url{https://cran.r-project.org/package=Kmedians}. In what follows, the centers initialization are generated from robust hierarchical clustering algorithm with \texttt{genieclust} package \citep{gagolewski2016genie}.

\subsection{Visualization of results with the package \texttt{Kmedians}}

In Section \ref{sec::choice}, we proved that the penalty shape is $a\sqrt{\frac{k}{n}}$ where $a$ is a constant to calibrate. To find the constant $a$, we will use the data-based calibration algorithm for penalization procedures that is explained at the end of section \ref{sec::choice}. This data-driven slope estimation method is implemented in CAPUSHE (CAlibrating Penalty Using Slope HEuristics) \citep{brault2011package} which is available in the $\textbf{\textsf{R}}$ package \texttt{capushe} \url{https://cran.r-project.org/package=capushe}. This proposed slope estimation method is made to be robust in order to preserve the eventual undesirable variations of criteria. More precisely, for a certain number of clusters $k$, the algorithm may be trapped by a local minima, which could create a “bad point” for the slope heuristic. The slope heuristic has therefore been designed to be robust to the presence of such points.

In what follows, we consider a random variable $X$ following a Gaussian Mixture Model with $k = 6$ classes where the mixture density function is defined as

$$ p(x) = \sum_{j=1}^{k} \pi_j \mathcal{N}(x|\mu_j,\operatorname{I}_d)      $$
with, $ \pi_j = \frac{1}{k} \hspace{3mm} \forall 1 \le j \le k, \hspace{3mm} \mu_j \sim \mathcal{U}_{10}  $ where $\mathcal{U}_{10}$  is the uniform law on the sphere of radius 10 and,
$$ \mathcal{N}(x|\mu,\operatorname{I}_d) = \frac{1}{\sqrt{(2\pi)^{d}}} \exp\left(-\frac{1}{2} \| x-\mu \|^2 \right). $$
In what follows, we consider $n = 3000$ i.i.d realizations of $X$ and $d = 5$. We first focus on some visualization of our slope method.

\begin{figure}[!h]\centering
\includegraphics[width=12cm, height=8cm]{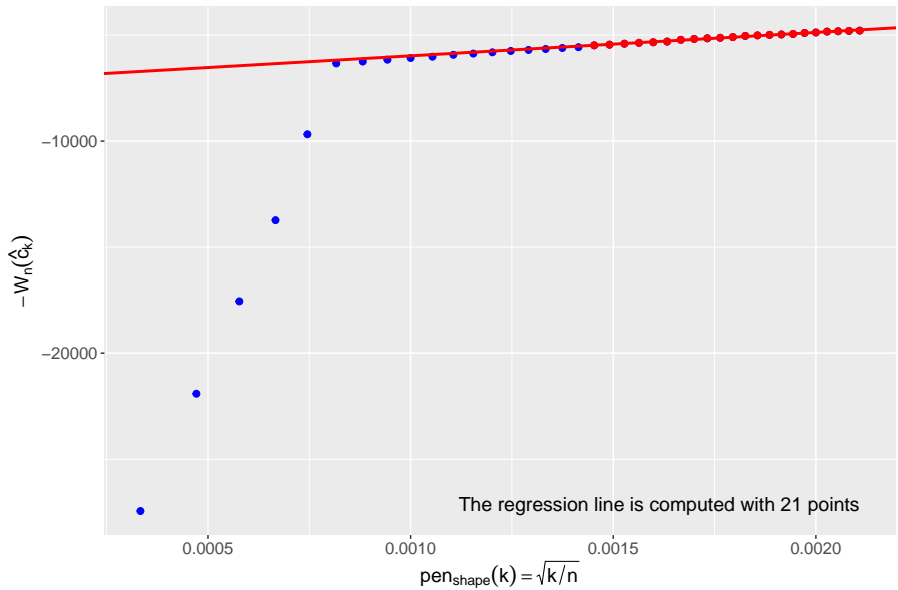}
\caption{\label{etiquette} Evolution of $-W_n(\hat{c}_{k})$ with respect to penalty shape: $\sqrt{k/n}$.}
\centering
\end{figure}

\begin{figure}[!h]\centering
\includegraphics[width=6cm,height=4cm]{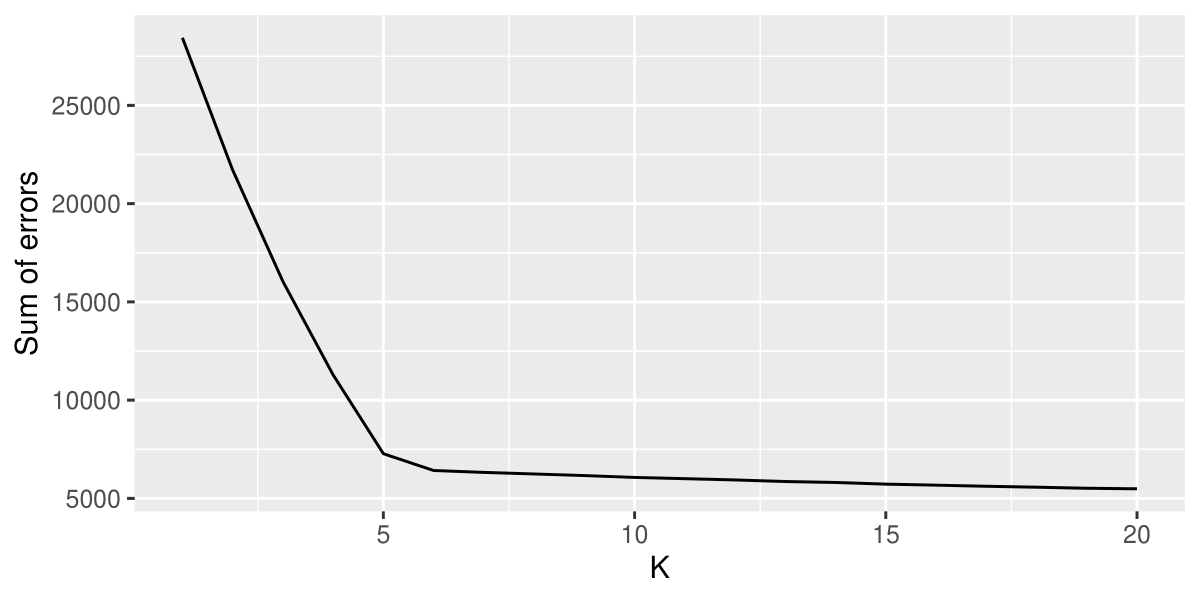}
\includegraphics[width=6cm,height=4cm]{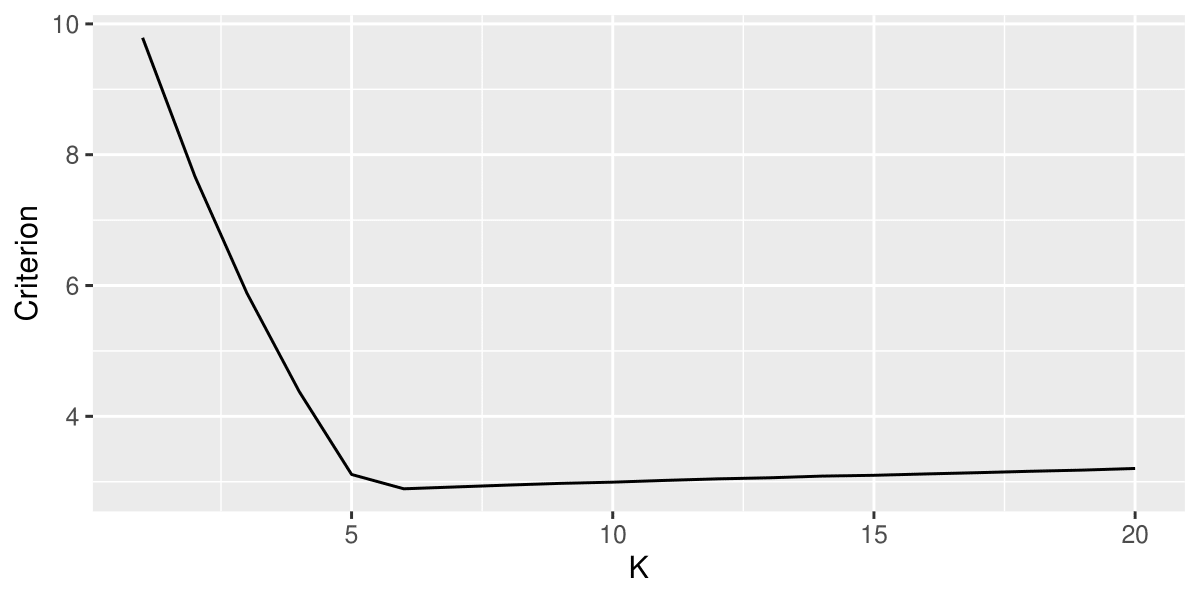}
\caption{\label{etiquette2}  Evolution of $W_n(\hat{c}_{k})$ (on the left) and $\text{crit}(k)$ (on the right) with respect to $k$.}
\centering
\end{figure}

To estimate $a \approx 2\hat{S}$ in the penalty function, it is sufficient to estimate $\hat{S}$, which is the slope of the red curve in Figure \ref{etiquette}. As shown in Figure \ref{etiquette}, the regression slope is estimated using the last 21 points, as it behaves like an affine function when $k$ is large. In Figure \ref{etiquette2} (left), two possible elbows are observed in the curve. Consequently, the elbow method suggests considering either 5 or 6 as the number of clusters. We would prefer to choose 5 since the elbow point at 5 is more pronounced compared to the one at 6. Therefore, this method is not ideal in this case.

Figures \ref{etiquette3} to \ref{etiquette5} represent the data as curves, which we call “profiles” (the x-label corresponds to the coordinates, and the y-label to the values of the coordinates), gathered by cluster and with the centers of the groups represented in red. We also show the first two principal components of the data using robust principal component analysis components (RPCA) \citep{cardot2017fast}. In Figure \ref{etiquette3}, we focus on the clustering obtained with the K-medians algorithm (“Offline” version) for non contaminated data.  In each cluster, the curves are close to each other and also close to the median, and the profiles differ from one cluster to another, meaning that our method separated well the 6 groups.
In order to visualize the robustness of the proposed method, we considered contaminated data with the law $Z = (Z_1,...,Z_5)$ where $Z_i$ are i.i.d, with $Z_i \sim \mathcal{T}_1$ where $\mathcal{T}_1$ is a Student law with one degree of freedom. Applying our method for selecting the number of clusters for K-medians algorithms, we selected the corrected number of clusters. Furthermore, the obtained groups, despite the presence of some outliers in each cluster, are coherent. Nevertheless, in the case of K-means clustering, the method found non homogeneous clusters, i.e. the method assimilates some far outliers as single clusters (see Figure \ref{etiquette5}). It's important to note that, in the case of contaminated data (Figures \ref{etiquette4} and \ref{etiquette5}), we only represented $95\%$ of the data to better visualize them. Then, in Figure, \ref{etiquette5}, Clusters 5, 7, 8, 11 and 12 are not visible since they are “far” outliers. 

\begin{figure}[!h]\centering
\includegraphics[width=6cm,height=4cm]{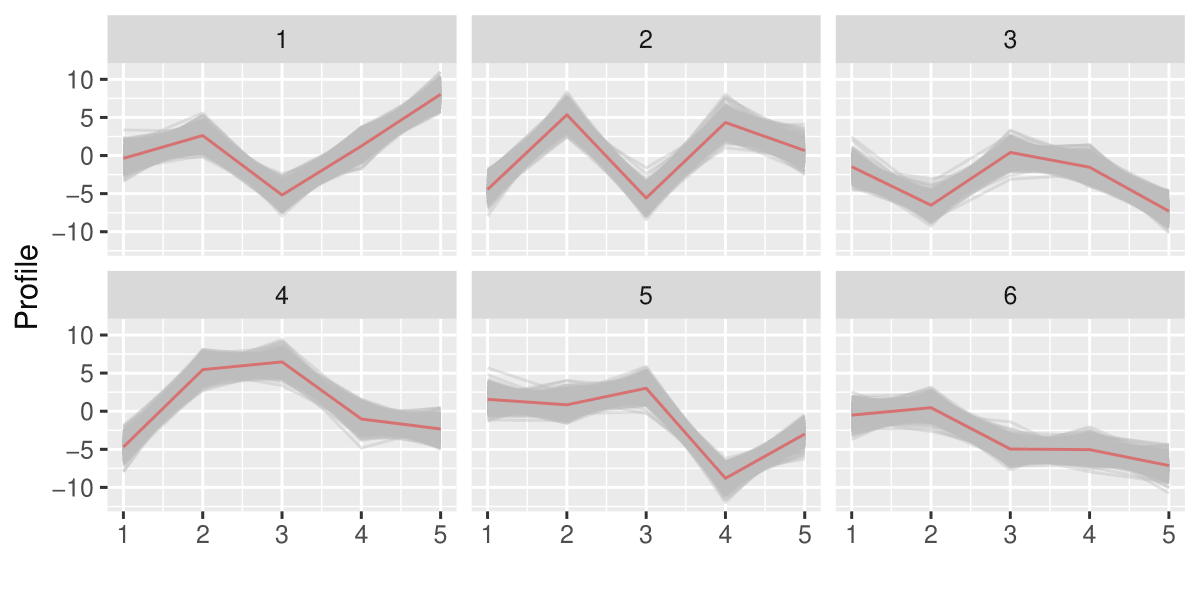} \quad \quad \quad
\includegraphics[width=6cm,height=4cm]{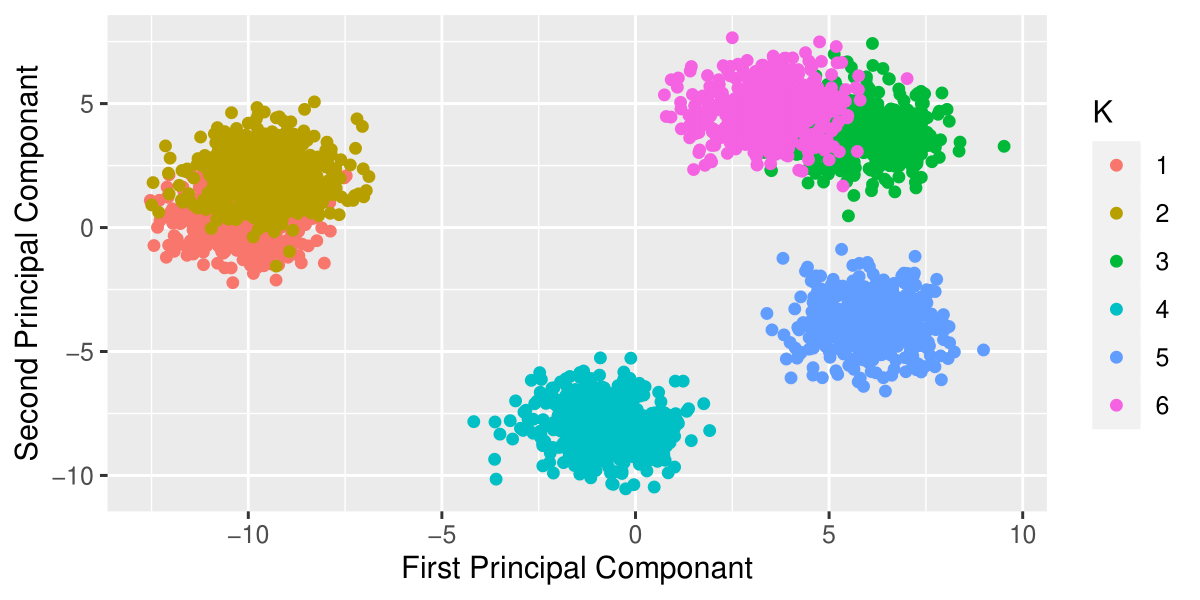}
\caption{\label{etiquette3} Profiles (on the left) and clustering via K-medians represented on the first two principal components (on the right) without contaminated data.}
\centering
\end{figure}

\begin{figure}[!h]\centering
\includegraphics[width=6cm,height=4cm]{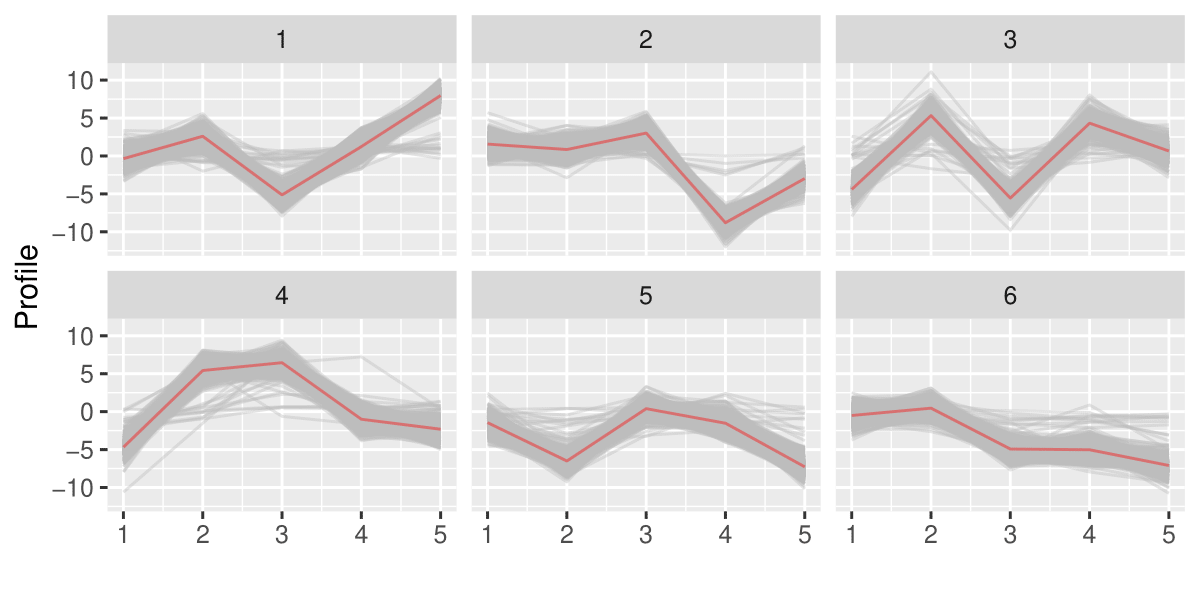}  \quad \quad \quad
\includegraphics[width=6cm,height=4cm]{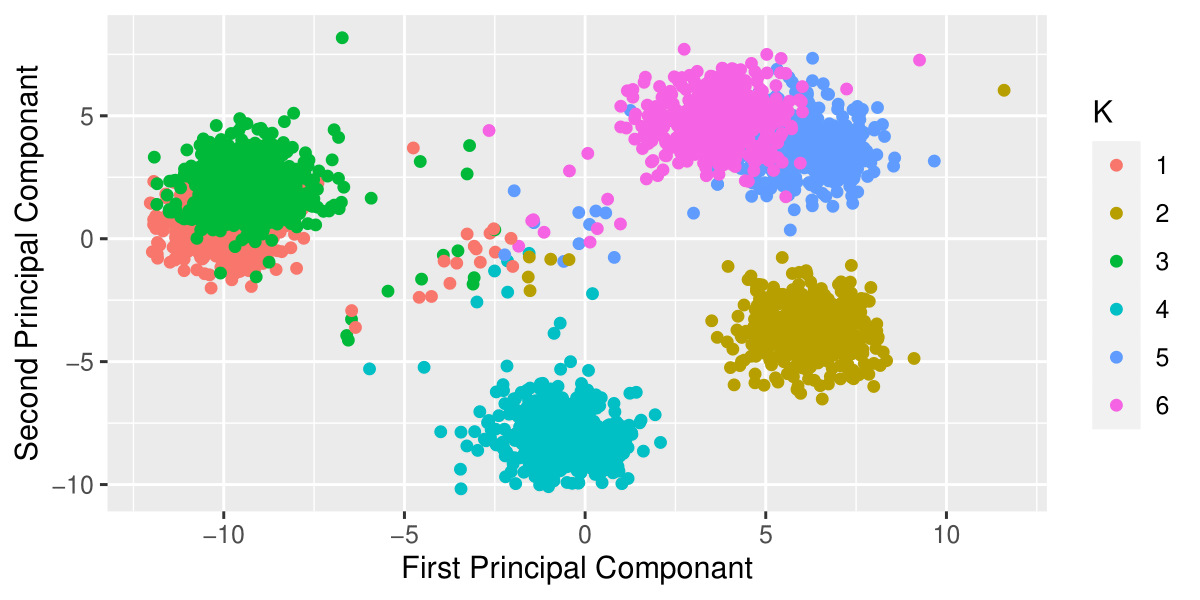}
\caption{\label{etiquette4} Profiles (on the left) and clustering via K-medians algorithm represented on the first two principal components (on the right) with $5\%$ of contaminated data. }
\centering
\end{figure}

\begin{figure}[!h]\centering
\includegraphics[width=6cm,height=4cm]{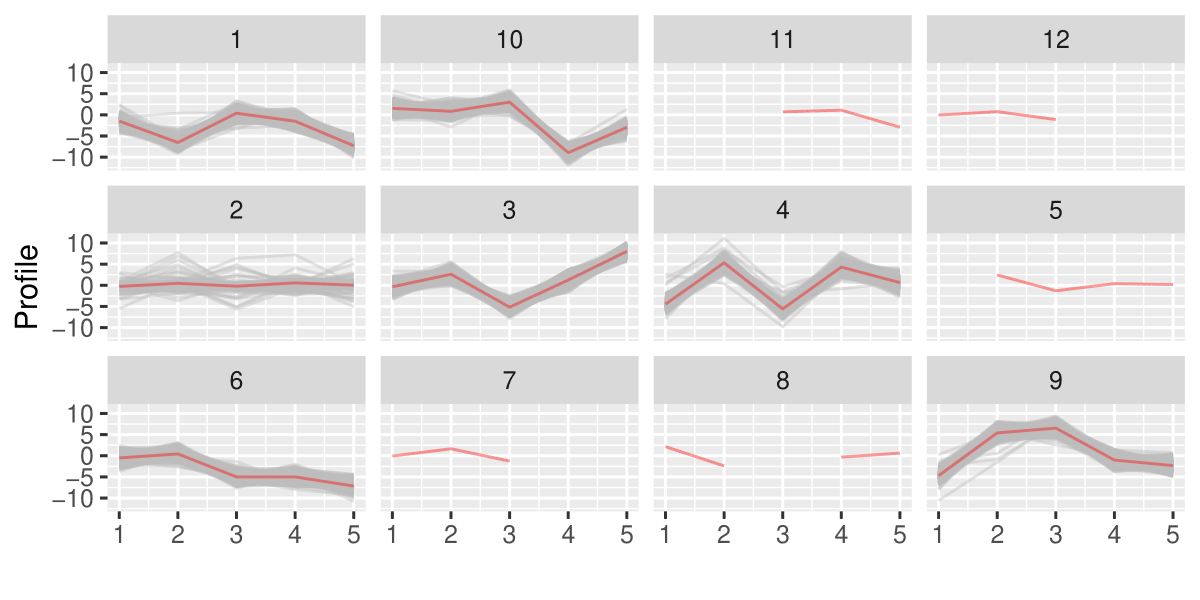}  \quad \quad \quad
\includegraphics[width=6cm,height=4cm]{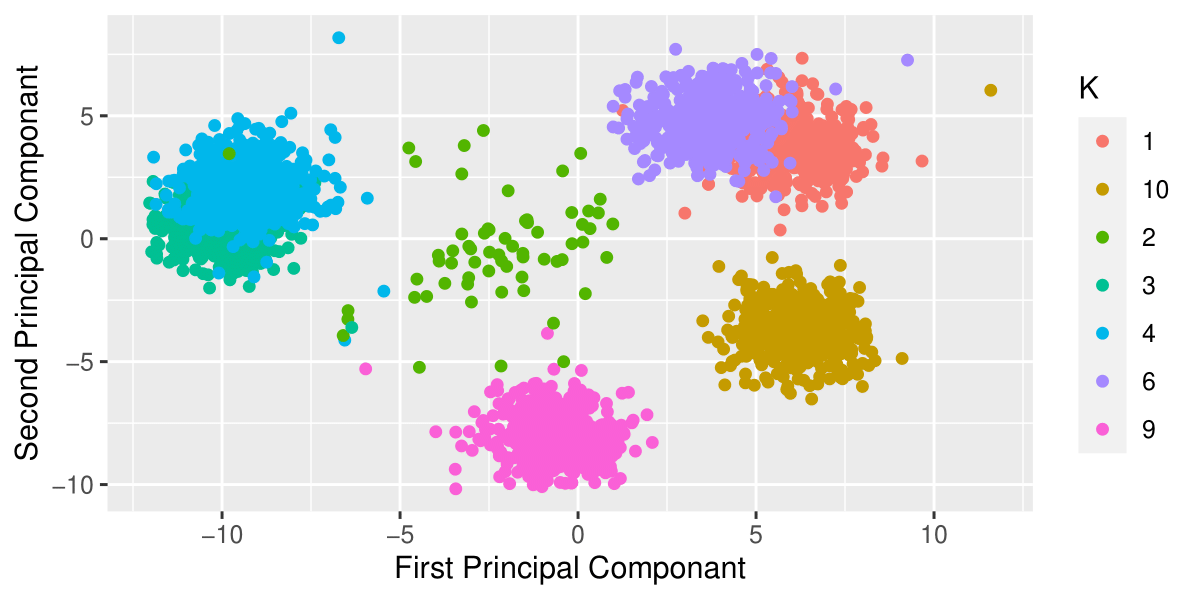}
\caption{\label{etiquette5} Profiles (on the left) and clustering via K-means algorithm represented on the first two principal components (on the right) with $5\%$ of contaminated data. }
\centering
\end{figure}

\newpage

\subsection{Comparison with Gap Statistic and Silhouette}
\label{sec2}

In what follows, we focus on the choice of the number of clusters and compare our results with different methods. For this, we generated some basic data sets in three different scenarios (see \cite{fischer2011number}) :\\
\textbf{(S1)  $4$ clusters in dimension $3$ : } The data are generated by Gaussian mixture centered at $(0, 0, 0)$, $(0, 2, 3)$, $(3, 0, -1)$, and $(-3, -1, 0)$ with variance equal to the identity matrix. Each cluster contains 500 data points.\\
\textbf{(S2)  $5$ clusters in dimension $4$ : } The data are generated by Gaussian mixture centered at $(0, 0, 0, 0)$, $(3, 5, -1, 0)$, $(-5, 0, 0, 0)$, $(1, 1, 6, -2)$ and $(1, -3, -2, 5)$ with variance equal to the identity matrix. Each cluster contains 500 data points.\\
\textbf{(S3)  $3$ clusters in dimension $2$ : } The data are generated by a Student mixture centered at $(0, 0)$, $(0, 6)$ and $(5, 3)$ with 2 degree of freedom. Each cluster contains 500 data points.\\

We applied three different methods for determining the number of clusters : the proposed slope method, Gap Statistic and Silhouette method. For each method, we use four clustering algorithms : K-medians (“Online”, “Semi-Online”, “Offline”) and K-means.
For each scenario, we contaminated our data with the law $Z = (Z_1,...,Z_d)$ where $Z_i$ are i.i.d, with $Z_i \sim \mathcal{T}_1$ where $\mathcal{T}_1$ is a Student law with 1 degree of freedom. Then, we evaluate our method for the different methods and scenarios by considering:
\begin{itemize}
\item N : The number of times we get the correct value of cluster in 50 repeated trials without contaminated data.
\item $\bar{k}$ : The average of number of clusters obtained over 50 trials without contaminated data.
\item $N_{0.1}$ : The number of times we get the correct value of cluster in 50 repeated trials with $10\%$ of contaminated data.
\item $\bar{k}_{0.1}$ : The average of number of clusters obtained over 50 trials with $10\%$ of contaminated data.
\end{itemize}

\begin{table}
\begin{center}
\begin{tabular}{|c|c|cccc|cccc|cccc|}
   \hline
   \multicolumn{2}{|c|}{Simulations} & \multicolumn{4}{c|}{S1}  & \multicolumn{4}{c|}{S2} & \multicolumn{4}{c|}{S3}\\
   \hline
    & Algorithms & $N$ & $\bar{k}$ & $N_{0.1}$  & $\bar{k}_{0.1}$ & $N$ & $\bar{k}$ & $N_{0.1}$ & $\bar{k}_{0.1}$ & $N$ & $\bar{k}$ & $N_{0.1}$ & $\bar{k}_{0.1}$ \\
   \hline
   \multirow{4}{*}{\rotatebox{90}{Slope}} & Offline & \textbf{50} & \textbf{4} & \textbf{50} & \textbf{4} & \textbf{50} &\textbf{5} &\textbf{50} & \textbf{5} & \textbf{50} & \textbf{3} & 49 & 3.04 \\
 
   & Semi-Online & 50 & 4 & 49 & 4.02 & \textbf{50} & \textbf{5} & 46 & 5.1 & \textbf{50} & \textbf{3} & 49 & 3.04 \\
 
   & Online & 48 & \textbf{4} & 42 & 4.1 & \textbf{50} & \textbf{5} & 40 & 5.2 & \textbf{50} & \textbf{3} & 49 & 3.04\\
 
   & K-means & \textbf{50} & \textbf{4} & 1 & 7.9 & $\textbf{50}$ & $\textbf{5}$ & 2 & 6.7 & 3 & 5.3 & 0 & 7.2 \\
   \hline
   \multirow{4}{*}{\rotatebox{90}{Gap}} & Offline & 6 & 1.7 & 0 & 1 & 47 & 4.8 & 2 & 1.2 & 0 & 1  & 0 & 1 \\
   
   & Semi-Online & 7 & 1.7 & 0 & 1 & 47 & 4.8 & 2 & 1.2 & 0 & 1  & 0 & 1 \\
 
   & Online & 8 & 2.4 & 0 & 1 & 47 & 4.8 & 2 & 1.2 & 0 & 1  & 0 & 1 \\
   
   & K-means & 0 & 1.2 & 0 & 1.2 & 12 & 2 & 0 & 1.3 & 0 & 1  & 0 & 1 \\
   \hline
   \multirow{4}{*}{\rotatebox{90}{Silhouette}} & Offline & 0 & 3 & 0 & 2.9 & 27 & 4.4 & 1 & 3.5 & $\textbf{50}$ & $\textbf{3}$ & 44 & 3.1 \\
   
   & Semi-Online & 0 & 3 & 0 & 2.9 & 24 & 4.4 & 1 & 3.5 & $\textbf{50}$ & $\textbf{3}$ & 43 & 3.1\\
 
   & Online & 0 & 3 & 2 & 3.2 & 22 & 4.5 & 2 & 4.5 & 49 & 3.02 & 29 & 3.5 \\
   
   & K-means & 0 & 3 & 7 & 3.2 & 20 & 4.5 & 0 & 6.7 & 3 & 5.9 & 2 & 7.2 \\
   \hline
\end{tabular}
\end{center} 
\caption{Comparison of the number of times we get the right value of clusters and the averaged selected number of clusters obtained with the different methods without contaminated data and with $10\%$ of contaminated data.} 
\end{table}

In case of well separated clusters as in the scenario (S2), the gap statistics method and silhouette method give competitive results. Nevertheless, for closer clusters, the slope method works much better than gap statistics and silhouette method as in the scenario (S1). 
The gap statistics method only works in scenario 2 and is ineffective in the presence of contamination. In closer cluster scenarios, it often predicts $1$ as the number of clusters.
The silhouette method performs moderately well in scenario 2 and very well in scenario 3, but it is globally not as competitive as the slope method, especially in cases of contaminated data. In scenarios 1 and 2 with slope method, Offline, Semi-Online, Online and K-means give better results but in cases of contamination, K-means crashes completely while the other three methods seem to be not too much sensitive. Furthermore, on non-Gaussian data (scenario 3), the K-means method does not work at all. In such cases, K-median clustering is often preferred over K-means clustering.

Overall, in every scenario, Offline, Semi-Online, Online K-medians with the slope method give very competitive results and in the case where the data are contaminated, they clearly outperform other methods, especially the Offline method. 

%As expected, in terms of efficiency, we find the order Offline, Semi-Online, Online since the sample size is moderate, but the Online algorithm is very competitive and is very cheap in term of computational calculus.

\subsection{Contaminated Data in Higher Dimensions}

We now focus on the impact of contaminated data on the selection of the number of clusters in K-medians clustering, particularly in higher dimensions. We compare our method with Gap Statistic and SigClust \citep{liu2008statistical, huang2015statistical, bello2023towards} in the Offline setting, as it yields competitive results, as noted in the previous section. Concerning SigClust, it is a method which enables to test whether a sample comes from a single Gaussian or several in high dimension. Then, starting from $k=k_0$, we test for all possible pairs of clusters whether the fusion of the two clusters comes from a single Gaussian or not. If the test rejected the hypothesis that the combined cluster is a single Gaussian for all fittings, the same procedure is repeated for $k+1$. If there is a fitting for which the test is not rejected, it is considered that these two clusters should be merged, and the procedure is stopped. The optimal number of clusters is then determined as $k_{\text{opt}} = k-1$. It is important to note that we did not compare with Gap Statistics, as it is computationally expensive, especially in high dimensions.

In this aim, we generate data using a Gaussian mixture model with 10 classes in 100 and 200 dimensions, where the centers of the classes are randomly generated on a sphere with radius $10$, and each class contains 100 data points. The data is contaminated with the law $Z = (Z_1,...,Z_d)$, where $d$ is the dimension (100 or 200 for each scenario), $Z_i$ are i.i.d, with two possible scenarios:

\begin{enumerate}
\item $Z_i \sim \mathcal{T}_1$,
\item $Z_i \sim \mathcal{T}_2$.  
\end{enumerate}

Here, $\mathcal{T}_m$ is the Student law with m degrees of freedom. In what follows, let us  denote by $\rho$ the proportion of contaminated data. In order to compare the different clustering results, we focus on the  Adjusted Rand Index (ARI)  \citep{rand1971objective, hubert1985comparing} which is a measure of similarity between two clusterings and which relies on taking into account the right number of correctly classified pairs.
We evaluate, for each scenario, the average number of clusters obtained over 50 trials and the average ARI evaluated only on uncontaminated data.

\begin{table}[!hbtp] 
\begin{center}
\begin{tabular}{|c|c|c|c|cccccc|}
   \hline
    &  & \multicolumn{1}{c}{$\rho$} & & 0 & 0.01 & 0.02 & 0.03 & 0.05 & 0.1 \\
   \hline
    \multirow{12}{*}{\rotatebox{90}{$d = 100$}} & \multirow{6}{*}{\rotatebox{90}{$Z_i \sim \mathcal{T}_1$}} & Our Method & \multirow{3}{*}{\rotatebox{90}{$\bar{k}$}} & $\textbf{10}$ & $\textbf{10}$ & $\textbf{10}$ & $\textbf{10}$ & $\textbf{10.2}$ & $\textbf{7.6}$  \\
 
    & & Silhouette  &  & $\textbf{10}$ & 8.7 & 5.9 & 4.2 & 3.6 & 2.6  \\
    
    & & SigClust &  & $\textbf{10}$ & 2.7 & 2.9 & 3.1 & 4.3 & 5.2  \\
 
    \cline{3-10}
    & & Our Method &  \multirow{3}{*}{\rotatebox{90}{ARI}} & $\textbf{1}$ & $\textbf{1}$ & $\textbf{1}$ &$\textbf{0.94}$ & $\textbf{0.91}$ & $\textbf{0.53}$ \\
    
    & & Silhouette &  & $\textbf{1}$ & 0.75 & 0.51 & 0.29 & 0.18 & 0.15 \\
    
    & & SigClust &  & $\textbf{1}$ & 0.22 & 0.28 & 0.29 & 0.31 & 0.46 \\
 
   \cline{2-10}
   
    & \multirow{6}{*}{\rotatebox{90}{$Z_i \sim \mathcal{T}_2$}} &   Our Method & \multirow{3}{*}{\rotatebox{90}{$\bar{k}$}} & $\textbf{10}$ & $\textbf{10}$ & $\textbf{10}$ & $\textbf{10}$ & $\textbf{10}$ & 10.9  \\
 
    & & Silhouette &  & $\textbf{10}$ & $\textbf{10}$ & $\textbf{10}$ & 10.1 & 9.8 & $\textbf{10.4}$  \\
    
    & & SigClust &  & $\textbf{10}$ & 5.3 & 4.5 & 4.2 & 4.4 & 4.2 \\
 
    \cline{3-10}
    & & Our Method &  \multirow{3}{*}{\rotatebox{90}{ARI}} & $\textbf{1}$ & $\textbf{1}$ & $\textbf{1}$ & $\textbf{1}$ & $\textbf{0.99}$ & $\textbf{0.99}$  \\
    
    & & Silhouette &  & $\textbf{1}$ & $\textbf{1}$ & $\textbf{0.99}$ & $\textbf{0.99}$ & $\textbf{0.99}$ & 0.98  \\
    
    & & SigClust  &  & $\textbf{1}$ & 0.52 & 0.36 & 0.35 & 0.34 & 0.32  \\
 
    \hline
    
    \multirow{12}{*}{\rotatebox{90}{$d = 200$}} & \multirow{6}{*}{\rotatebox{90}{$Z_i \sim \mathcal{T}_1$}} & Our Method & \multirow{3}{*}{\rotatebox{90}{$\bar{k}$}} & $\textbf{10}$ & $\textbf{10}$ & $\textbf{10}$ & $\textbf{9.6}$ & $\textbf{9.6}$ & $\textbf{7.3}$  \\
 
    & & Silhouette  &  & $\textbf{10}$ & 6 & 2.9 & 2.9 & 2.5 & 3.1  \\
    
    & & SigClust &  & 9.2 & 2.7 & 3.4 & 4.3 & 5 & 6.5 \\
 
    \cline{3-10}
    & & Our Method &  \multirow{3}{*}{\rotatebox{90}{ARI}} & $\textbf{1}$ & $\textbf{1}$ & $\textbf{1}$ & $\textbf{0.89}$ & $\textbf{0.76}$ & $\textbf{0.53}$  \\
     
    & & Silhouette &  & $\textbf{1}$ & 0.39 & 0.18 & 0.17 & 0.16 & 0.21  \\
        
    & & SigClust &  & 0.97 & 0.22 & 0.28 & 0.34 & 0.42 & 0.48  \\

   \cline{2-10}
   
    & \multirow{6}{*}{\rotatebox{90}{$Z_i \sim \mathcal{T}_2$}} &   Our Method & \multirow{3}{*}{\rotatebox{90}{$\bar{k}$}} & $\textbf{10}$ & $\textbf{10}$ & $\textbf{10}$ & $\textbf{10}$ & $\textbf{10}$ & $\textbf{10}$  \\
 
    & & Silhouette &  & $\textbf{10}$ & $\textbf{10}$ & $\textbf{10}$ & $\textbf{10}$ & 10.1 & 10.2  \\
    
    & & SigClust &  & 9.2 & 4.6 & 4.3 & 3.9 & 3.8 & 2.7 \\
 
    \cline{3-10}
    & & Our Method &  \multirow{3}{*}{\rotatebox{90}{ARI}} & $\textbf{1}$ & $\textbf{1}$ & $\textbf{1}$ & $\textbf{1}$ & $\textbf{1}$ & $\textbf{1}$  \\
    
    & & Silhouette &  & $\textbf{1}$ & $\textbf{1}$ & $\textbf{1}$ & $\textbf{1}$ & 0.99 & 0.99  \\
    
    & & SigClust  &  & 0.97 & 0.37 & 0.35 & 0.32 & 0.32 & 0.22  \\
 
    \hline

\end{tabular}
\end{center} 
\caption{Comparison of the selected number of clusters and the averaged ARI obtained obtained using different methods with respect to the proportion of contaminated data for $Z_i \sim \mathcal{T}_1$ and  $Z_i \sim \mathcal{T}_2$. } 
\end{table}

We observe that in the case of non-contamination, we obtain similar results across all methods. However, in the presence of contamination, our method consistently performs well, while others struggle to identify an appropriate number of clusters.  With a Student distribution contamination of 1 degree of freedom, our method excels in terms of both the number of clusters and the ARI. The results with a Student distribution contamination of 2 degrees of freedom are comparable to those obtained using the Silhouette method.

In summary, our method demonstrates remarkable robustness in the face of contaminated data, making it a strong choice for clustering in higher dimensions. The comparison with the Silhouette, Gap Statistic, and SigClust in the offline setting reaffirms the effectiveness of our approach, especially when computational efficiency is a critical factor in high-dimensional data.

\subsection{An illustration on real data}  %Real-life data

We will first briefly discuss the data we used for clustering, which was provided by Califrais, a company specializing in developing environmentally responsible technology to optimize logistics flows on a large scale. Our goal is to build a Recommender System that is designed to suggest items individually for each user based on their historical data or preferences. In this scenario, the clustering algorithms can be employed to identify groups of similar customers where each group consists of customers share similar features or properties. It is crucial to perform robust clustering in order to develop an effective Recommender System.

The dataset includes information on 508 customers, including nine features that represent the total number of products purchased in each of the following categories: Fruits, Vegetables, Dairy products, Seafood, Butcher, Deli, Catering, Grocery, and Accessories and equipment. Therefore, we have a sample size of $n=508$ and a dimensionality of $d=9$. To apply clustering, we will determine the appropriate number of clusters using the proposed method. Before applying our method, we normalize our data using RobustScaler. This  removes the median and scales the data according to the Interquartile Range, which is the range between the 1st quartile and the 3rd quartile.

\begin{figure}[!h]
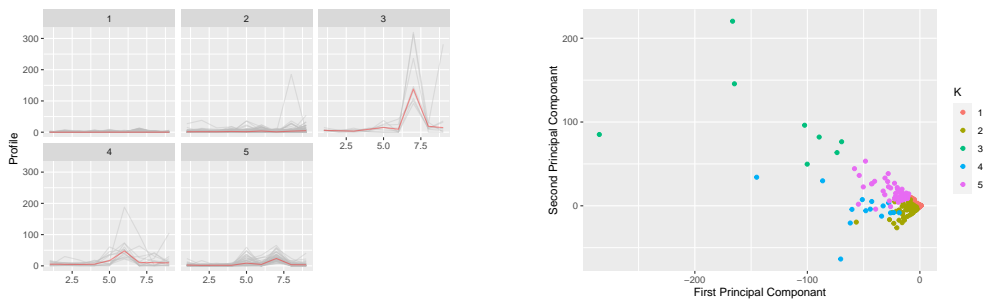
\centering
\includegraphics[width=6cm,height=4cm]{cali_slope1.png}  \quad \quad \quad
\includegraphics[width=6cm,height=4cm]{cali_slope2.png}
\caption{\label{cali_slope} Califrais data: Profiles (on the left) and clustering with Slope method represented on the first two principal components (on the right). }
\centering
\end{figure}

\begin{figure}[!h]\centering
\includegraphics[width=6cm,height=4cm]{cali_silhouette1.png}  \quad \quad \quad
\includegraphics[width=6cm,height=4cm]{cali_silhouette2.png}
\caption{\label{cali_silhouette} Califrais data: Profiles (on the left) and clustering with Silhouette method represented on the first two principal components (on the right). }
\centering
\end{figure}

\begin{figure}[!h]
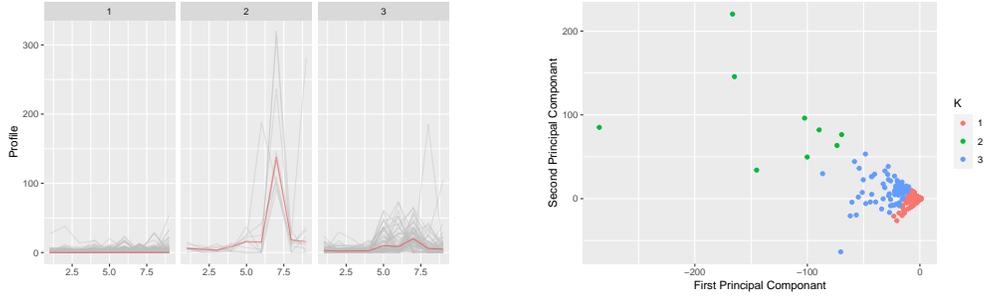
\centering
\includegraphics[width=6cm,height=4cm]{cali_gap1.png}  \quad \quad \quad
\includegraphics[width=6cm,height=4cm]{cali_gap2.png}
\caption{\label{cali_gap} Califrais data: Profiles (on the left) and clustering with Gap Statistics method represented on the first two principal components (on the right). }
\centering
\end{figure}

We plotted the profiles of the clusters obtained using our Slope method, Silhouette and Gap Statistic in Figures \ref{cali_slope}, \ref{cali_silhouette}, and \ref{cali_gap}. We observe that our method indicates 5 clusters, while the Gap Statistic suggests 3 clusters, and Silhouette suggests 2 clusters. Regarding the Silhouette method, the second cluster obtained is not homogeneous, as seen in Figure \ref{cali_silhouette}. We obtain 3 clusters with the Gap Statistic method. The important thing to note is that the Gap Statistic method separates the second cluster obtained by Silhouette into two clusters (Cluster 2 and Cluster 3). However, in the third cluster of Gap Statistic (Figure \ref{cali_gap}), homogeneity is still not achieved. In Figure \ref{cali_slope}, it can be seen that the clusters generated by our slope method are more or less homogeneous.
To establish a connection with the simulations conducted in Section \ref{sec2}, for example, in scenario (S1), we observed that Silhouette and Gap Statistics failed to find the correct number of clusters when the clusters are closer. This is reflected here, as the behavior of clients does not change significantly, resulting in close clusters.
To provide an overview of our clusters, the first cluster represents customers who regularly consume products from all categories. The third cluster consists of customers who frequently engage with catering products. Clusters 2, 4, and 5 correspond to customers who consume significant amounts of Butcher, Deli, and Catering products at different levels, as depicted in the figure \ref{cali_slope}.
%With only 508 customers, interpreting these clusters is feasible. However, when dealing with millions of data points, identifying homogeneous clusters becomes challenging, thus impeding cluster analysis. Therefore, it is crucial to obtain robust clusters, especially to have a method to determine the number of clustering.

\subsection{Conclusion}

The proposed penalized criterion, calibrated with the help of the slope heuristic method, consistently gives competitive results for selecting the number of clusters in K-medians, even in the presence of outliers, outperforming other methods such as Gap Statistics, Silhouette, and SigClust. Notably, our method demonstrates excellent performance even in high dimensions. Among the three K-medians algorithms, Offline, Semi-Online, and Online, their performances are generally analogous, with Offline being slightly better. However, for large sample sizes, one may prefer the Online K-medians algorithm in terms of computation time. As discussed in Section \ref{sec::framework}, it is recommended to use the Offline algorithm for moderate sample sizes, the Semi-Online algorithm for medium sample sizes, and the Online algorithm for large sample sizes. In our real-life data illustration, our proposed method consistently produces more robust clusters and a more suitable number of clusters compared to other methods.

In conclusion, our paper presents a robust and efficient approach for selecting clusters in K-medians, demonstrating superior performance even in challenging scenarios. The findings provide practical recommendations for algorithm selection based on sample size, reinforcing the applicability of our proposed method in real-world clustering scenarios.

\subsection*{Acknowledgement}

The authors wish to thank Califrais for providing the real-life data and Raphaël Carpintero Perez for the data preprocessing work.

\section{Proofs}\label{sec::proofs}

\subsection{Some definitions and lemma}

First, we provide some definitions and lemmas that are useful to prove Theorems \ref{theo1} and \ref{theo2}.\\

$\textbf{Definitions  : }$
\begin{itemize}
\item Let $(S,p)$ be a totally bounded metric space. For any $F \subset S$ and $\epsilon > 0$ the $\epsilon$-covering number $N_p(F,\epsilon)$ of F is defined as the minimum number of closed balls with radius $\epsilon$ whose union covers $F$.
\item Let $(S,p)$ be a totally bounded metric space. For any $F \subset S$,\\ $\text{diam}(F) = \sup \left\{ p(x, y) : x, y \in F \right\}$.

\item A Family $\left\{ T_s : s\in S \right\} $ of zero-mean random variables indexed by the metric space (S, p) is called subgaussian in the metric p if for any $\lambda > 0$ and $s, s' \in S$ we have 
$$ \mathbb{E} \left[e^{\lambda( T_s - T_{s'} )}  \right]   \le e^{\frac{\lambda^2 p(s, s')^2}{2}}.$$
\item The Family $\left\{ T_s : s\in S \right\} $ is called sample continuous if for any sequence $s_1,s_2..,\in S$ such that $ s_j \rightarrow s \in S$ we have $T_{s_j} \rightarrow T_s$ with probability one.
\end{itemize}

\begin{lem}[\cite{hoeffding1994probability}]\label{lem1}
Let $Y_1,..,Y_n$ are independent zero-mean random variables such that $ a\le Y_i\le b, i = 1,...,n$, then for all $\lambda > 0$,
$$ \mathbb{E} \left[e^{\lambda( \sum_{i=1}^{n}Y_i  )}  \right]   \le e^{\frac{\lambda^2 n(b-a)^2}{8}}. $$
\end{lem}

\begin{lem}[\cite{cesa1999minimax}, Proposition 3]\label{lem2}
If $\left\{ T_s : s\in S \right\} $ is subgaussian  and sample continuous in the metric p, then
$$ \mathbb{E} \left[\sup_{s \in S} T_s  \right]   \le  12 \int_{0}^{\text{diam}(S)/2} \sqrt{\ln{N_{p}(S,\epsilon)}  }  d\epsilon.   $$
\end{lem}
\begin{lem}[\cite{bartlett1998minimax}, Lemma 1]\label{lem3}
Let $S(0,R)$ denote the closed d-dimensional sphere of radius $R$ centered at $0$. Let $\epsilon > 0$ and $N(\epsilon)$ denote the cardinality of the minimum $\epsilon$ covering of $S(0,R)$, that is, $N(\epsilon)$ is the smallest integer $N$ such that there exist points $\left\{y_1,...,y_N\right\} \subset S(0,R)$ with the property
$$ \sup_{x \in S(0,R)} \min_{1 \le i \le N } \left\|  x - y_i \right\| \le \epsilon. $$
Then, for all $\epsilon \le 2R$ we have 
$$ N(\epsilon) \le \left(\frac{4R}{\epsilon}\right)^{d}. $$
\end{lem}
\begin{lem}\label{lem4}
For any $0 < \epsilon \le 2R$ and $ k \ge 1$, the covering number of  $S_k$ in the metric 
$$ p(c,c') = \sup_{\| x \| \le R} \left\{ \big| \min_{j = 1,..,k}  \left\|  x - c_j \right\| - \min_{j = 1,..,k}  \| x - c'_j \|  \hspace{1mm} \big|  \right\}$$
is bounded as 
$$N_p(S_k,\epsilon) \le \left(\frac{4R}{\epsilon}\right)^{kd}.$$
\end{lem}
\begin{proof}[$\textbf{Proof of the Lemma \ref{lem4} : }$]
Let $0 < \epsilon \le 2R$, by Lemma \ref{lem3} there exists a $\epsilon$-covering set of points $\left\{y_1,...,y_N\right\}  \subset S(0,R)$ with $N \le \left(\frac{4R}{\epsilon}\right)^d$.\\
Since, we have $N^k$ ways to choose k codepoints from a set of N points $\left\{y_1,...,y_N\right\}$,
that implies $$N_p(S_k,\epsilon) \le \left(\frac{4R}{\epsilon}\right)^{kd}.$$
For any codepoints $\left\{c_1,...,c_k\right\}$ which are contained in $S(0,R)$, there exists a set of codepoints such that $\left\|  c_j - c'_j \right\| \le \epsilon$ for all j.
\\
Let us first show 
$$  \min_{j = 1,..,k}  \left\|  x - c_j \right\| - \min_{j = 1,..,k}  \left\|  x - c'_j \right\| \le \epsilon.    $$
In this aim, let us consider $ q \in \arg\min_{j = 1,..,k}   \left\|  x - c_j \right\|$, then 
$$  \min_{j = 1,..,k}  \left\|  x - c'_j \right\| - \min_{j = 1,..,k}  \left\|  x - c_j \right\|  \le \left\|  x - c'_q \right\| - \left\|  x - c_q \right\| \le
\left\|  c_q - c'_q \right\| \le \epsilon.    $$ 
In the same way, considering $ q' \in \arg\min_{j = 1,..,k}   \left\|  x - c'_j \right\|$ , we show
$$  \min_{j = 1,..,k}  \left\| x - c_j \right\| - \min_{j = 1,..,k}  \left\|  x - c'_j \right\|  \le \left\|  x - c_{q'} \right\| - \left\|  x - c'_{q'} \right\| \le
\left\|  c_{q'} - c'_{q'} \right\| \le \epsilon.   $$
So,
$$ \big| \min_{j = 1,..,k}  \left\|  x - c_j \right\| - \min_{j = 1,..,k}  \left\|  x - c'_j \right\|  \hspace{1mm} \big| \le \epsilon$$
for any codepoints $\left\{c_1,...,c_k\right\}$ which are contained in $S(0,R)$, there exists a set of codepoints $\left\{c'_1,...,c'_k\right\}$ such that 
$$ \big| \min_{j = 1,..,k}  \left\|  x - c_j \right\| - \min_{j = 1,..,k}  \left\|  x - c'_j \right\|  \hspace{1mm} \big| \le \epsilon.$$
\end{proof}

\begin{lem}[\cite{mcdiarmid1989method}, \cite{massart2007concentration} : Theorem 5.3]\label{lem5}
If $X_1,...X_n$ are independent random variables and $\mathcal{F}$ is a finite or countable class of real-valued functions such that $a \le f \le b$ for all $ f \in \mathcal{F}$, the if $ Z = \sup_{f \in \mathcal{F}} \sum_{i=1}^{n} (f(X_i) - \mathbb{E} \left[f(X_i) \right] )  $, we have, for every $\epsilon > 0$,
$$ \mathds{P} \left[Z - \mathbb{E} \left[Z \right] \ge \epsilon \right] \le \exp \left(-\frac{2\epsilon^2}{n(b-a)^2}\right).  $$   
\end{lem}

\subsection{Proof of Theorem \ref{theo1}}

The proof of the Theorem \ref{theo1} is inspired by the proof of Theorem 3 in \cite{linder2000training}. 

\begin{proof}
For any $c \in S_k$, let $ T_n^{(c)}  = \frac{n}{2}(W(c) - W_n(c))   $
$ = \frac{1}{2} \sum_{i=1}^{n}  (\mathbb{E} \left[\min_{j = 1,..,k}  \left\|  X_i - c_j \right\|  \right] - \min_{j = 1,..,k}  \left\|  X_i - c_j \right\|).     $
\\So $$ \mathbb{E} \left[\sup_{c \in S_k } ( W(c) - W_n(c) ) \right]  = \frac{2}{n} \mathbb{E}\left[\sup_{c \in S_k } T_n^{(c)} \right]. $$
Let us first demonstrate that the family of random variables $\left\{ T_n^{(c)} : c\in S_k \right\} $ is subgaussian and sample continuous in a suitable metric.
For any $c, c' \in S_k $ define \\
$$ p(c,c') = \sup_{\| x \| \le R} \left\{ \big| \min_{j = 1,..,k}  \left\|  x - c_j \right\| - \min_{j = 1,..,k}  \left\| x - c'_j \right\|  \hspace{1mm} \big|  \right\}, $$
and $ p_n(c, c') = \sqrt{n}p(c, c') $,
$p_n$ is a metric on $S_k$. 
Since we have:
\begin{alignat*}{2}
 \big| T_n^{(c)} - T_n^{(c')} \big| &= \frac{n}{2}\big|W(c) - W(c') + W_n(c') - W_n(c)\big| \\  
& \le \frac{n}{2}\left(\big|W(c) - W(c')\big| + \big|W_n(c') - W_n(c)\big|\right)  \\
& \le np(c, c') = \sqrt{n}p_n(c, c'), 
\end{alignat*}

the family $\left\{ T_n^{(c)} : c\in S_k \right\} $  is then sample continuous in the metric $p_n$.
To show that $\left\{ T_n^{(c)} : c\in S_k \right\} $  is subgaussian in $p_n$, let
\\
$$ Y_i = \frac{1}{2}\left(W(c) - \min_{j = 1,..,k}  \left\|  x - c_j \right\|\right) - \frac{1}{2}\left(W(c') - \min_{j = 1,..,k}  \left\|  x - c'_j \right\|\right). $$
Then 
$$ T_n^{(c)} - T_n^{(c')} = \sum_{i=1}^{n}Y_i, $$
where $Y_i$ are independent, have zero mean, and 
$$ \big|Y_i \big| \le \frac{1}{\sqrt{n}}p_n(c, c').  $$
By Lemma \ref{lem1}, we obtain 
$$ \mathbb{E} \left[e^{\lambda( T_n^{(c)} - T_n^{(c')} )}  \right]   \le e^{\frac{\lambda^2 p_n(c, c')^2}{2}}. $$
So, $\left\{ T_n^{(c)} : c\in S_k \right\} $  is subgaussian in $p_n$.
As the family $\left\{ T_n^{(c)} : c\in S_k \right\} $  is subgaussian and sample continuous in $p_n$, Lemma \ref{lem2} gives
$$ \mathbb{E} \left[\sup_{c \in S_k} T_n^{(c)}  \right]   \le  12 \int_{0}^{\text{diam}(S_k)/2} \sqrt{\ln{N_{p_n}(S_k,\epsilon)}  }  d\epsilon.   $$
Since $p_n(c, c') = \sqrt{n}p(c, c')$, by Lemma \ref{lem4} with the metric $p_n$, for all $\epsilon \leq 2R\sqrt{n}$ we obtain
$$N_{p_n}(S_k,\epsilon) \le \left(\frac{4R\sqrt{n}}{\epsilon}\right)^{kd},$$
and as $\text{diam}(S_k) := \sup \left\{ p_n(c, c') : c, c' \in S_k \right\} = \sqrt{n} \sup \left\{ p(c, c') : c, c' \in S_k \right\} \le \sqrt{n}2R$,
\begin{alignat*}{2}
\mathbb{E} \left[\sup_{c \in S_k} T_n^{(c)}  \right]   &\le  \frac{24}{n} \int_{0}^{\sqrt{n}R} \sqrt{\ln\left(\left({\frac{4R\sqrt{n}}{\epsilon} }\right)^{kd}\right) }  d\epsilon   \\
&=  \frac{24\sqrt{kd}}{n} \int_{0}^{\sqrt{n}R} \sqrt{\ln\left({\frac{4R\sqrt{n}}{\epsilon} }\right)}  d\epsilon.  
\end{alignat*}
Considering $ x = \frac{\epsilon}{4R\sqrt{n}} $, we obtain,
$$ \mathbb{E} \left[\sup_{c \in S_k} T_n^{(c)}  \right]   \le  \frac{24\sqrt{kd}}{n} \int_{0}^{\frac{1}{4}} 4R\sqrt{n}\sqrt{\ln\left({\frac{1}{x} }\right) }  dx.   $$
Applying Jensen's inequality to the concave function $f(x) = \sqrt{x}$ :
\begin{alignat*}{2}
 \mathbb{E} \left[\sup_{c \in S_k} T_n^{(c)}  \right]  & \le 24R \sqrt{\frac{kd}{n}} \sqrt {\int_{0}^{\frac{1}{4}} 4\ln\left({\frac{1}{x} }\right)   dx }  \\
 &= 24R \sqrt{\frac{kd}{n}} \sqrt{ 1 + \ln{4} } \\
 &\le 48R \sqrt{\frac{kd}{n}}, 
\end{alignat*}
where we used that  $ \int \ln{x} = x\ln{x} - x$ and $\ln{4} \le 3$.\\
Thus,
$$  \mathbb{E} \left[\sup_{c \in S_k } \left\{ W(c) - W_n(c) \right\}  \right]   \le 48R\sqrt{\frac{kd}{n}}. $$
\\
\\
\end{proof}

\subsection{Proof of Theorem \ref{theo2}}

Theorem \ref{theo2} is an adaptation of Theorem 8.1 in  \cite{massart2007concentration}  and Theorem 2.1 in  \cite{fischer2011number}.

\begin{proof}
By definition of $ \tilde{c} , $ for all $ k,  1 \le k \le n $ and $ c_k \in S_k $, we have:
$$ W_n(\tilde{c}) + \text{pen}(\hat{k})  \le  W_n(c_k) + \text{pen}(k)   $$
\begin{equation}
\label{2}
  W(\tilde{c})  \le  W_n(c_k)  + W(\tilde{c}) - W_n(\tilde{c}) + \text{pen}(k) - \text{pen}(\hat{k}).   
\end{equation}

Consider nonnegative weights $ \left\{ x_l \right\}_{1 \le l \le n} $ such that $ \sum_{l=1}^{n} e^{-x_l} =   \Sigma $ and let z $>$ 0.
\\
Applying Lemma \ref{lem5} with $f(x) =  \frac{1}{n}  \min_{j = 1,..,l}  \left\|  x - c_j \right\|$,   $a = 0$ and $b = \frac{2R}{n}$   for all $ l,  1 \le l \le n $ and all  $ \epsilon_l  > 0 $
$$\mathds{P} \left[ \sup_{c \in S_l } ( W(c) - W_n(c) ) - \mathbb{E} \left[\sup_{c \in S_l } ( W(c) - W_n(c) ) \right]   \ge \epsilon_l   \right]  \le \exp \left(- \frac{n\epsilon_l^2}{2R^2} \right).  $$
It follows that for all l, taking $\epsilon_l = 2R\sqrt{\frac{x_l + z}{2n} } $
\\
$$\mathds{P} \left[ \sup_{c \in S_l } ( W(c) - W_n(c) ) \ge \mathbb{E} \left[\sup_{c \in S_l } ( W(c) - W_n(c) ) \right]  + 2R\sqrt{\frac{x_l + z}{2n} }    \right]  \le e^ { -x_l -z }.  $$
Thus, we have
$$\mathds{P} \left[   \bigcap_{l=1}^{n}   \sup_{c \in S_l } ( W(c) - W_n(c) ) \le \mathbb{E} \left[\sup_{c \in S_l } ( W(c) - W_n(c) ) \right]  + 2R\sqrt{\frac{x_l + z}{2n} }   \right]   $$
$$ = 1 - \mathds{P} \left[   \bigcup_{l=1}^{n}   \sup_{c \in S_l } ( W(c) - W_n(c) ) \ge \mathbb{E} \left[\sup_{c \in S_l } ( W(c) - W_n(c) ) \right]  + 2R\sqrt{\frac{x_l + z}{2n} }   \right]   \ge 1 - \Sigma e^{-z}. $$
Considering $ Z_l = \mathbb{E} \left[\sup_{c \in S_l } ( W(c) - W_n(c) ) \right] $, let us show if we have for all $1 \le l \le n $, 
$$ \sup_{c \in S_l } ( W(c) - W_n(c) ) \le Z_l  + 2R\sqrt{\frac{x_l + z}{2n} } $$
then, 
$$ W(\tilde{c}) \le W_n(c_k)  + 2R\sqrt{\frac{z}{2n} } + \text{pen}(k). $$
We suppose that we have
\begin{equation}
\label{3}
  \sup_{c \in S_l } ( W(c) - W_n(c) ) \le Z_l  + 2R\sqrt{\frac{x_l + z}{2n} }   \hspace{1cm}  \forall  1 \le l \le n. 
\end{equation}
Particularly it's true for $l = \hat{k}$, we have also
$ W(\tilde{c}) - W_n(\tilde{c}) \le \sup_{c \in S_{\hat{k}} } ( W(c) - W_n(c) )   $ and $ \sqrt{a+b} \le  \sqrt{a} + \sqrt{b} $ $ \forall a,b \ge 0 $.
By combining this result with (2) and (3), we get
\begin{alignat*}{2}
W(\tilde{c}) &\le  W_n(c_k)  +  \sup_{c \in S_{\hat{k}} } ( W(c) - W_n(c) ) + \text{pen}(k) - \text{pen}(\hat{k})  \\
& \le  W_n(c_k)  + Z_{\hat{k}} + 2R\sqrt{\frac{x_{\hat{k}} }{2n} } + 2R\sqrt{\frac{z}{2n} } + \text{pen}(k) - \text{pen}(\hat{k}). 
\end{alignat*}
With the help of Theorem \ref{theo2}, we have $ Z_k \le 48R\sqrt{\frac{kd}{n}} $ for all $ k,  1 \le k \le n $ 
and if we have $ \text{pen}(k)  \ge R\left( 48\sqrt{\frac{kd}{n} }  +  2\sqrt{\frac{x_k}{2n} }\right) $
\begin{alignat*}{2}
W(\tilde{c}) &\le  W_n(c_k)  + 48R\sqrt{\frac{\hat{k}d}{n}} + 2R\sqrt{\frac{x_{\hat{k}} }{2n} } + 2R\sqrt{\frac{z}{2n} } + \text{pen}(k) - R\left( 48\sqrt{\frac{\hat{k}d}{n} }  +  2\sqrt{\frac{x_{\hat{k}}}{2n} }\right) \\
&= W_n(c_k)  + 2R\sqrt{\frac{z}{2n} } + \text{pen}(k), 
\end{alignat*}
which shows that 
$$ W(\tilde{c}) \le W_n(c_k)  + 2R\sqrt{\frac{z}{2n} } + \text{pen}(k). $$

Thus 

$$ \mathds{P} \left[  W(\tilde{c})  \le  W_n(c_k)  + 2R\sqrt{\frac{z}{2n} } + \text{pen}(k)  \right]  $$
$$ \ge \mathds{P} \left[   \bigcap_{l=1}^{n}   \sup_{c \in S_l } ( W(\mu , c) - W(\mu_n , c) ) \le \mathbb{E} \left[\sup_{c \in S_l } ( W(c) - W_n(c) ) \right]  + 2R\sqrt{\frac{x_l + z}{2n} }   \right] \ge 1 - \Sigma e^{-z}. $$
We get
$$ \mathds{P} \left[  W(\tilde{c}) - W_n(c_k) - \text{pen}(k)  \ge 2R\sqrt{\frac{z}{2n} }  \right] \le \Sigma e^{-z}    $$
$$ \mathds{P} \left[  \frac{\sqrt{2n}}{ 2R} (W(\tilde{c}) - W_n(c_k) - \text{pen}(k))  \ge \sqrt{z}  \right] \le \Sigma e^{-z}    $$
or, setting $ z = u^2 $,
$$ \mathds{P} \left[  \frac{\sqrt{2n}}{ 2R} (W(\tilde{c}) - W_n(c_k) - \text{pen}(k))  \ge u  \right] \le \Sigma e^{-u^2}    $$

\begin{alignat*}{2}
 \mathbb{E} \left[  \frac{\sqrt{2n}}{ 2R} (W(\tilde{c}) - W_n(c_k) - \text{pen}(k))_+   \right]  &= \int_{0}^{\infty} \mathds{P} \left[  \frac{\sqrt{2n}}{ 2R} (W(\tilde{c}) - W_n(c_k) - \text{pen}(k))_+  \ge u  \right] du  \\
& \le \int_{0}^{\infty} \mathds{P} \left[  \frac{\sqrt{2n}}{ 2R} (W(\tilde{c}) - W_n(c_k) - \text{pen}(k))  \ge u  \right] du \\
& \le  \Sigma \int_{0}^{\infty} e^{-u^2} du  = \Sigma \frac{\sqrt{\pi}}{2}.   
\end{alignat*}
\\
We get
$$ \mathbb{E} \left[ (W(\tilde{c}) - W_n(c_k) - \text{pen}(k))_+   \right]  \le \Sigma R \sqrt{\frac{\pi}{2n} }.  $$ 
Since $ \mathbb{E} \left[W_n(c_k)  \right]  = W(c_k) $, we have :
\\
$$ \mathbb{E} \left[W(\tilde{c})  \right] \le W(c_k) + \text{pen}(k) + \Sigma R \sqrt{\frac{\pi}{2n} }.  $$ 

$$  \mathbb{E} \left[W(\tilde{c})  \right] \le \inf_{ 1 \le k \le n , c_k \in S_k } \left\{W(c_k) + \text{pen}(k) \right\} + \Sigma R \sqrt{\frac{\pi}{2n} }.  $$
\end{proof}

\subsection{Proof of Proposition \ref{prop}}
\begin{proof}
If $k \le 2^d$, we have $4Rk^{-1/d} \ge 4R2^{-1} = 2R$. 
Thus, $ W(c) \le 2\sqrt{d} \le 4\sqrt{d}k^{-1/d} $ for any vector quantizer with codebook c.\\
Otherwise, let $\epsilon = 4Rk^{-1/d}$. Then $\epsilon \le 2R$ and by Lemma \ref{lem3} there exists a set of points $\left\{y_1,...,y_k\right\}  \subset S(0,R)$ that $\epsilon$-covers $S(0,R)$. A quantizer with the codebook $c = \left\{y_1,...,y_k\right\}$ verifies :
$$ W(c) \le \epsilon \le 4Rk^{-1/d}. $$
That concludes 
$$ \inf_{c \in S_k }W(c) \le 4Rk^{-1/d}. $$
\end{proof}

\bibliographystyle{apalike}
\bibliography{ref}

\end{document}